\documentclass[10pt]{amsart}
\usepackage{amssymb}
\usepackage{amsbsy}
\usepackage{amscd}
\usepackage[mathscr]{eucal}
\usepackage{verbatim}
\usepackage{version}
\usepackage{color}
\usepackage[matrix,arrow]{xy}
\usepackage{graphicx}

\def\cal{\mathcal}
\def\Bbb{\mathbb}

\newenvironment{NB}{
\color{red}{\bf NB}. \footnotesize 
}{}
\excludeversion{NB}
\newenvironment{NB2}{
\color{blue}{\bf NB}. \footnotesize
}{}

\newcommand{ \Supp}{\operatorname{Supp}}
\newcommand{\Ext}{\operatorname{Ext}}
\newcommand{\Hom}{\operatorname{Hom}}

\newcommand{\rk}{\operatorname{rk}}

\newcommand{\NS}{\operatorname{NS}}

\newcommand{\Pic}{\operatorname{Pic}}
\newcommand{\ch}{\operatorname{ch}}

\newcommand{\Hilb}{\operatorname{Hilb}}

\newcommand{\Coh}{\operatorname{Coh}}

\newcommand{\Amp}{\operatorname{Amp}}

\font\b=cmr10 scaled \magstep5
\def\bigzerou{\smash{\lower1.7ex\hbox{\b 0}}}

\numberwithin{equation}{section}
\theoremstyle{plain}
 \newtheorem{thm}{Theorem}[section]
 \newtheorem{lem}[thm]{Lemma}
 \newtheorem{prop}[thm]{Proposition}
 \newtheorem{cor}[thm]{Corollary}
 
 \newtheorem{prob}[thm]{Problem}
\theoremstyle{definition}
 \newtheorem{defn}[thm]{Definition}
 
\theoremstyle{remark}
 \newtheorem{rem}[thm]{Remark}
 \newtheorem{ex}[thm]{Example}
 
\begin{document}

\title{A note on the moduli of stable sheaves on elliptic ruled surfaces.}
\author{K\={o}ta Yoshioka}
\address{Department of Mathematics, Faculty of Science,
Kobe University,
Kobe, 657, Japan
}
\maketitle

\section{Introduction}
Let $\pi:X \to C$ be a ${\Bbb P}^1$-bundle over a smooth projective curve $C$
of genus $g$. Let $C_0$ be a minimal section of $\pi$ 
with $(C_0^2)=-e$ and $f$ a fiber of $\pi$.
Then $K_X \equiv -2C_0-(2-2g+e)f$. 
If $((K_X+f) \cdot H)<0$. then $\Ext^2(E,E(-f))=0$ for all $\mu$-semi-stable 
sheaves. Thus $E$ is a prioritary sheaf in the sense of Walter.
Then Walter \cite{Walter}
proved the irreducibility by using deformation theory
of prioritary sheaves. 

In this note, we shall slightly generalize the irreducibility of the moduli spaces.
Thus we shall prove the irreducibility of the moduli spaces under the assumption
$(K_X \cdot H)<0$. 
We note that $(K_X \cdot H)<0$ for any ample divisor $H$, if $g=0,1$.
Indeed $-K_X$ is numerically equivalent to an effective divisor.
If $g=0$, then $-K_X-f$ is also effective, and hence the irreducibility is a consequence of \cite{Walter}. 
On the other hand if $g=1$, 
then $-K_X-f$ is numerically equivalent to an effective divisor
if and only if $e\geq 1$. So we are interested in the case where $e=0,-1$.
We also study the existence condition of the $\mu$-semistable sheaves
in these cases.

For a coherent sheaf on $X$, $\rk E$ denotes the rank of $E$ and we set
$\tau(E):=(\rk E,c_1(E),\chi(E))$. 
Let ${\cal M}_H({\bf e})^{\mu ss}$ be the moduli stack of $\mu$-semi-stable sheaves $E$ with 
$\tau(E)={\bf e}$.
Then we have the following result.
\begin{thm}[{Proposition \ref{prop:irred}, \ref{prop:1}}]\label{thm:(HK_X)}
Let $X$ be an elliptic ruled surface and $H(x):=C_0+ xf$ be an ample divisor on $X$.
\begin{NB}
Assume that $(K_X \cdot H)<0$. 
\end{NB}
Then for 
${\bf e}=(r,\xi,\chi) \in {\Bbb Z}_{>0} \times \NS(X) \times {\Bbb Z}$,
with $0 \leq (\xi \cdot f)<r$, we have the following.
\begin{enumerate}
\item[(1)]
${\cal M}_{H(x)}(r,\xi,\chi)^{\mu ss}$ is irreducible.
\item[(2)]
${\cal M}_{H(x)}(r,\xi,\chi)^{\mu ss} \ne \emptyset$ if and only if
$r^2 \Delta({\bf e}) \geq r_1 (r-r_1)(x-\frac{e}{2})$, where $r_1=(\xi \cdot f)$.
For the definition of $\Delta({\bf e})$, see \eqref{eq:Delta2}.
\item[(3)]
Assume that $r>1$.
Let $E \in {\cal M}_{H(x)}(r,\xi,\chi)^{\mu ss}$ be a general member. Then
$E$ is a $\mu$-stable locally free sheaf such that
$E_{|\pi^{-1}(t)}$ is a rigid vector bundle for a general $t \in C$. 
\end{enumerate}
\end{thm}

Let $X$ be an elliptic ruled surface and
assume that $-K_X$ is nef, that is, $e=0,-1$.
Then $(K_X^2)=0$ and there is an adjacent chamber for ${\bf e}=(r,\xi,\chi)$.
We denote the moduli stack by ${\cal M}_{-K_X}({\bf e})^{\mu ss}$.
\begin{cor}\label{cor:main}
Let $X$ be an elliptic ruled surface and
assume that $-K_X$ is nef, that is, $e=0,-1$. Then
${\cal M}_{-K_X}(r,\xi,\chi)^{\mu ss} \ne \emptyset$ if and only if $\Delta({\bf e}) > 0$
or $r \mid (\xi \cdot f)$ and $\Delta({\bf e})=0$.
\end{cor}

\begin{NB}
We use Theorem \ref{thm:(HK_X)}.
If ${\cal M}_{-K_X}(r,\xi,\chi)^{\mu ss} \ne \emptyset$, then Bogomolov inequality implies
$\Delta({\bf e}) \geq 0$. 
Moreover if $\Delta({\bf e})=0$, then $r \mid (\xi \cdot f)$ by Lemma \ref{lem:.Delta}.
Conversely we assume that $\Delta({\bf e}) \geq 0$.
Since $x-e/2>0$ is sufficiently small, if $\Delta({\bf e})>0$ or $r_1=0$, then
the claim follows.

If $r_1=0$, then obviously the claim holds.
If $r_1>0$, then $\Delta>0$ by Bogomolov inequality.
Since $x-e/2>0$ is sufficiently small, our claim holds.
\end{NB}

\begin{NB}
\begin{cor}
Let $X$ be an elliptic ruled surface. Then $(H(x) \cdot K_X)<0$ for any ample divisor $H(x)$.
Thus the claims of Theorem \ref{thm:(HK_X)} hold
for any $H(x)$. 
\end{cor}
\end{NB}


\section{Basic results on moduli spaces}

\subsection{Notation.}
In this note, instead of using equivariant cohomology of suitable quot-schemes as in \cite{chamber}, 
we use moduli stacks of coherent sheaves.

Let ${\cal M}({\bf e})$ be the stack of coherent sheaves $E$ with $\tau(E)={\bf e}$.
Let ${\cal M}_H({\bf e})^{\mu ss}$ be the open substack of ${\cal M}({\bf e})$ consisting of
$\mu$-semi-stable sheaves
and ${\cal M}_H({\bf e})^{\mu s}$ the open substack consisting of
$\mu$-stable sheaves.
Let ${\cal M}_H({\bf e})^{ss}$ (resp. ${\cal M}_H({\bf e})^{s}$)
be the substack of ${\cal M}_H({\bf e})^{\mu ss}$ consisting of
semi-stable sheaves (resp. stable sheaves).
Let ${\cal M}_H({\bf e})^{\mu s}_0$ be the open substack of 
${\cal M}_H({\bf e})^{\mu s}$  consisting of $\mu$-stable locally free sheaves.

Let $\overline{M}_H({\bf e})$ be the moduli space of $S$-equivalence classes of semi-stable sheaves
$E$ with $\tau(E)={\bf e}$ and $M_H({\bf e})$ the open subscheme consisting of stable sheaves.
Let $M_H({\bf e})_0^{\mu s}$ be the moduli space of $\mu$-stable locally free sheaves.

For a coherent sheaf $E$ with $\rk E=r \ne 0$,
we set
\begin{equation}\label{eq:Delta1}
\begin{split}
\mu(E):=& \frac{c_1(E)}{r},\\
\Delta(E):=&
\frac{2rc_2(E)-(r-1)(c_1(E)^2)}{2r^2}\\
=& \frac{-2r \chi(E)+2r^2 \chi({\cal O}_X)-r(c_1(E) \cdot K_X)+(c_1(E)^2)}{2r^2}.
\end{split}
\end{equation}
For ${\bf e}=(r,\xi,\chi)$, we also set
\begin{equation}\label{eq:Delta2}
\mu({\bf e}):=\frac{\xi}{r},\; \Delta({\bf e}):=\frac{-2r \chi+2r^2 \chi({\cal O}_X)-r(\xi \cdot K_X)+(\xi^2)}{2r^2}.
\end{equation}


For coherent sheaves $E,F$ with $\rk E,\rk F>0$,
Riemann-Roch theorem implies
\begin{equation}
\chi(E,F)=\rk E \rk F \left( \frac{((\mu(F)-\mu(E))^2)-((\mu(F)-\mu(E)) \cdot K_X)}{2}+\chi({\cal O}_X) -\Delta(E)-\Delta(F)\right).
\end{equation}

\subsection{Stack of filtrations}

For ${\bf e}_i:=(r_i,\xi_i,\chi_i)$ $(1 \leq i \leq s)$,
let ${\cal F}_H({\bf e}_1,{\bf e}_2,...,{\bf e}_s)$ be the stack of filtrations
\begin{equation}\label{eq:filt}
0 \subset F_1 \subset F_2 \subset \cdots \subset F_s=E
\end{equation}
such that $E_i:=F_i/F_{i-1} \in {\cal M}_H({\bf e}_i)^{\mu ss}$ $(1 \leq i \leq s)$ 
 and 
$\Ext^2(E_j,E_i)=0$ ($i < j$).

\begin{prop}\label{prop:filt}
$\dim {\cal F}_H({\bf e}_1,{\bf e}_2,...,{\bf e}_s)=2r^2 \Delta({\bf e})-r^2 \chi({\cal O}_X)
+\sum_{i<j}\chi({\bf e}_i,{\bf e}_j)$.
\end{prop}

\begin{proof}
By the proof of \cite[Lem. 5.2]{K-Y} and Proposition \ref{prop:dim}, we see that
\begin{equation}
\begin{split}
\dim {\cal F}_H({\bf e}_1,{\bf e}_2,...,{\bf e}_s) =& \sum_i \dim {\cal M}_H({\bf e}_i)^{\mu ss}-
\sum_{i>j} \chi({\bf e}_i,{\bf e}_j)\\
=& -\chi({\bf e},{\bf e})+\sum_{i<j}\chi({\bf e}_i,{\bf e}_j).
\end{split}
\end{equation}
\end{proof}

\begin{rem}
\begin{enumerate}
\item
Assume that 
$$
\frac{(\xi_1 \cdot H)}{r_1}>\frac{(\xi_2 \cdot H)}{r_2}> \cdots>
\frac{(\xi_s \cdot H)}{r_s}.
$$
Then \eqref{eq:filt} is the Harder-Narasimhan filtration of $E$. 
\begin{NB}
Then ${\cal F}_H({\bf e}_1,{\bf e}_2,...,{\bf e}_s)$ parameterizes Harder-Narasimhan filtrations
of $E \in {\cal M}({\bf e})$.
\end{NB}
In this case, ${\cal F}_H({\bf e}_1,{\bf e}_2,...,{\bf e}_s) \to {\cal M}({\bf e})$
is an immersion. 
\item
Assume that 
$$
\frac{(\xi_1 \cdot H)}{r_1}=\frac{(\xi_2 \cdot H)}{r_2}= \cdots=
\frac{(\xi_s \cdot H)}{r_s}.
$$
Then \eqref{eq:filt} is a Jordan-H\"{o}lder filtration of $E$.
In this case, the image of ${\cal F}_H({\bf e}_1,{\bf e}_2,...,{\bf e}_s) \to {\cal M}({\bf e})$
parameterizes properly $\mu$-semi-stable sheaves having Jordan-H\"{o}lder filtrations
\eqref{eq:filt}. 
\end{enumerate}
\end{rem}

\subsection{Walls and chambers}
Let $\Amp(X)$ be the ample cone of $X$.

 Let ${\cal F}({\bf e})$ 
be the set of subsheaves $F \subset E$
 which satisfy 
\begin{enumerate}
\item[(1)]
$\tau(E)={\bf e}$, 
\item[(2)]
 $0<\rk F <\rk E$,
\item[(3)]
 $\Delta(F),\Delta(E/F) \geq 0$ and 
\item[(4)]
$\mu(E)-\mu(F) \ne 0$ and 
there is an element $H \in \Amp(X)$ with $((\mu(E)-\mu(F))\cdot H)=0$.
\end{enumerate}
\begin{defn}\label{defn:2-1}
For an element $0 \subset F
 \subset E$ of ${\cal F}({\bf e})$, we define a wall $W^F$ by
$$
W^F:=\{ H \in \Amp(X)|((\mu(E)-\mu(F)) \cdot H)=0 \}.
$$
$\cup_F W^F$ is locally finite.
We shall call chamber a  connected component of $\Amp(X) \setminus \cup_F W^F$.
\end{defn}

\begin{lem}
Assume that $(K_X \cdot H)<0$.
Let $E$ and $F$ be $\mu$-semi-stable sheaves of positive rank
satisfying
$\frac{(c_1(E) \cdot H)}{\rk E}=\frac{(c_1(F) \cdot H)}{\rk F}$.
Then $\Ext^2(E,F)=0$.
\end{lem}

\begin{proof}
Since $(-K_X \cdot H)>0$, we have
$\frac{(c_1(E(K_X)) \cdot H)}{\rk E}<\frac{(c_1(F) \cdot H)}{\rk F}$.
Hence 
$$
\Ext^2(E,F)=\Hom(F,E(K_X))^{\vee}=0.
$$
\end{proof}

\begin{prop}\label{prop:dim}
\begin{enumerate}
\item[(1)]
Let ${\cal M}$ be an irreducible component of ${\cal M}_H({\bf e})^{\mu ss}$.
Then 
$$
\dim {\cal M} \geq 2r^2 \Delta({\bf e})-r^2 \chi({\cal O}_X).
$$
\item[(2)]
Assume that $(K_X \cdot H)<0$. Then ${\cal M}_H({\bf e})^{\mu ss}$ is smooth of
$$
\dim {\cal M}_H({\bf e})^{\mu ss}=2r^2 \Delta({\bf e})-r^2 \chi({\cal O}_X).
$$
\end{enumerate}
\end{prop}

\begin{proof}
By the deformation theory of coherent sheaves,
$$
\dim {\cal M} \geq \dim \Ext^1(E,E)-\dim \Ext^2(E,E)-\dim \Hom(E,E)=-\chi(E,E)
$$
for $E \in {\cal M}$.
Hence we get the first claim.
By $(K_X \cdot H)<0$, $\Ext^2(E,E)=0$ for all $E \in {\cal M}_H({\bf e})^{\mu ss}$.
Hence ${\cal M}_H({\bf e})^{\mu ss}$ is smooth.
\end{proof}

\begin{thm}[{Beauville \cite{B}, Markman \cite{Ma}}]\label{thm:Beauville}
Let $X$ be a smooth projective surface and $H$ an ample divisor on $X$
such that $(K_X \cdot H)<0$.
If $\gcd(r,\xi,\chi)=1$, then
$M_H(r,\xi,\chi)$ is a smooth, projective and irreducible.
\end{thm}

\begin{proof}
By $(K_X \cdot H)<0$,
$\Ext^2(E,F)=\Hom(F,E(K_X))^{\vee}=0$ for all
$E,F \in M_H(r,\xi,\chi)$.
Hence \cite[Thm. 8]{Ma} holds.
Then the proof of \cite[Cor. 10]{Ma} implies the claim.
\end{proof}

\begin{cor}\label{cor:Beauville}
Let $X$ be a smooth projective surface and $H$ an ample divisor on $X$
such that $(K_X \cdot H)<0$.
Then $M_H(r,\xi,\chi)_0^{\mu s}$ is irreducible.
\end{cor}

\begin{proof}
For the proof, 
we shall apply Theorem \ref{thm:Beauville} 
to the moduli space on a blow-up of $X$.
By using a result of Nakashima \cite[Thm. 1]{Nakashima} on the relations of moduli spaces on
$X$ and the blow-up, we shall prove the claim. 

Let $p:\widetilde{X} \to X$ be the blow-up of $X$ at a point $x$ and $C_1$ the exceptional divisor.
For $E \in M_H(r,\xi,\chi)_0^{\mu s}$,
let $E'$ be a locally free sheaf on $\widetilde{X}$ fitting in an exact sequence
\begin{equation}
0 \to p^*(E) \to E' \to {\cal O}_{C_1}(-1) \to 0.
\end{equation}
Then $E'$ is a $\mu$-stable locally free sheaf with respect to $H_n:=n p^*(H)-C_1$ for sufficiently large 
$n$ \cite[Prop. 2.5]{Nakashima}.
We note that $c_1(E')=\pi^*(\xi)+C_1$ and $\chi(E')=\chi(E)=\chi$.
Since $c_1(E')$ is primitive and $(H_n \cdot K_{\widetilde{X}})=n(H \cdot K_X)+1<0$ for $n \gg 0$,
Theorem \ref{thm:Beauville} implies $M_{H_n}(r,\pi^*(\xi)+C_1,\chi)$ is irreducible.
Since \cite[Thm 2.6]{Nakashima} holds without assuming the universal family, 
we get the irreducibility of $M_H(r,\xi,\chi)_0^{\mu s}$.
\end{proof}

\begin{rem}
If there is not a universal family, then the Grassmanian bundle in \cite[Thm 2.6]{Nakashima}
is \'{e}tale locally trivial. 
\end{rem}

\section{A study of chambers}

\subsection{Several dimension estimates}

Let $C_0+x f$ be an ample ${\Bbb Q}$-divisor on $X$.
Then $x>\frac{e}{2}$.

\begin{lem}\label{lem:(DK_X)}
Assume that $x>\frac{e}{2}$,
$(D \cdot (C_0+x f))=0$ and $(D \cdot f)>0$.
Then
$(D \cdot K_X) \geq 0$.
\end{lem}

\begin{proof}
We set $D=aC_0+b f$. 
By $(D \cdot (C_0+x f))=0$ and $(D \cdot f) \geq 0$, we get
$0=-ae+b+xa$ and $a \geq 0$.
Hence $b=-a(x-e)$.
\begin{equation}
\begin{split}
(D \cdot K_X)=& ((aC_0+bf)\cdot (-2C_0-(2-2g+e)f))\\
=& 2ae-2b-(2-2g+e)a\\
=& a(2x-e+2g-2) \geq 0.
\end{split}
\end{equation}
\end{proof}

\begin{lem}\label{lem:pss}
Assume that $(K_X \cdot H(x))<0$.
\begin{equation}
\dim ({\cal M}_{H(x)}({\bf e})^{\mu ss} \setminus {\cal M}_{H(x-\epsilon)}({\bf e})^{\mu ss})<
2r^2 \Delta({\bf e})+r^2(g-1).
\end{equation}
\end{lem}

\begin{proof}
For $E \in {\cal M}_{H(x)}({\bf e})^{\mu ss} \setminus {\cal M}_{H(x-\epsilon)}({\bf e})^{\mu ss}$,
let $0 \subset F_1 \subset F_2 \subset \cdots \subset F_s=E$ be the Harder-Narasimhan filtration
of $E$ with respect to
$H(x-\epsilon)$.
We set $E_i:=F_i/F_{i-1}$.
Then 
$$
\frac{(c_1(E_i) \cdot (H(x-\epsilon)))}{\rk E_i}>\frac{(c_1(E_j) \cdot (H(x-\epsilon)))}{\rk E_j},\;
\frac{(c_1(E_i) \cdot H(x))}{\rk E_i}=\frac{(c_1(E_j) \cdot H(x))}{\rk E_j},
$$
for $i<j$.
Then
$$
\frac{(c_1(E_i) \cdot f)}{\rk E_i}<\frac{(c_1(E_j) \cdot f)}{\rk E_j}
$$
for $i<j$.

We set ${\bf e}_i:=\tau(E_i)$.
$\mu_i:=\mu({\bf e}_i)$, $\Delta_i:=\Delta({\bf e}_i)$.
Then $((\mu_j-\mu_i) \cdot f)>0$.
By Lemma \ref{lem:(DK_X)}, we get $((\mu_j-\mu_i) \cdot K_X) \geq 0$.
Hence $-\chi({\bf e}_i,{\bf e}_j)>0$ for $i<j$. By Proposition \ref{prop:filt}, we get
\begin{equation}
\dim {\cal F}_{H(x-\epsilon)}({\bf e}_1,{\bf e}_2,...,{\bf e}_s)<2r^2 \Delta({\bf e})+r^2(g-1).
\end{equation}

\begin{NB}
Let $\Gamma$ be the set of ${\bf e}_i:=(r_i,\xi_i,\chi_i)$ such that (1)
$\Delta({\bf e}_i) \geq 0$ for $1 \leq i \leq s$, (2)
$\sum_i {\bf e}_i={\bf e}$ and (3)
$$
\frac{((\xi_i \cdot (H(x-\epsilon)))}{r_i}>\frac{(\xi_j \cdot (H(x-\epsilon)))}{r_j},\;
\frac{((\xi_i \cdot H(x))}{r_i}=\frac{(\xi_j \cdot H(x))}{r_j},
$$
for $i<j$.
Then 
$$
\cup_{({\bf e}_1,{\bf e}_2,...,{\bf e}_s) \in \Gamma} 
{\cal F}_{H(x-\epsilon)}({\bf e}_1,{\bf e}_2,...,{\bf e}_s)=
{\cal M}_{H(x)}({\bf e})^{\mu ss} \setminus {\cal M}_{H(x-\epsilon)}({\bf e})^{\mu ss},.
$$
and hence
$$
\max_{({\bf e}_1,{\bf e}_2,...,{\bf e}_s) \in \Gamma} 
\dim  {\cal F}_{H(x-\epsilon)}({\bf e}_1,{\bf e}_2,...,{\bf e}_s)=
\dim ({\cal M}_{H(x)}({\bf e})^{\mu ss} \setminus {\cal M}_{H(x-\epsilon)}({\bf e})^{\mu ss}).
$$
\end{NB}

\end{proof}

\begin{prop}\label{prop:N(x)}
Assume that $(K_X \cdot H(x))<0$.
Let $N(x)$ be the set of irreducible components ${\cal M}$ of ${\cal M}({\bf e})$
such that ${\cal M} \cap {\cal M}_{H(x)}({\bf e})^{\mu ss} \ne \emptyset$.
Then $N(x) \subset N(y)$ for $x \geq y$.
\end{prop}

\begin{proof}
For an ample divisor $H(x)$, we take a sufficiently small $\epsilon>0$. 
Since ${\cal M}_{H(x \pm \epsilon)}({\bf e})^{\mu ss}$ are open substacks of ${\cal M}_{H(x)}({\bf e})^{\mu ss}$,
$N(x \pm \epsilon) \subset N(x)$.
By Proposition \ref{prop:dim} and Lemma \ref{lem:pss}, 
we also have $N(x)=N(x-\epsilon)$.
Therefore the claim holds.  
\end{proof}

\begin{lem}\label{lem:Delta}
Let $E$ be a $\mu$-semi-stable sheaf of rank $r$ on $X$ with respect to an ample divisor
such that $\Delta(E)=0$.
Then $E$ is $\mu$-semi-stable for any ample divisor and
$r \mid (c_1(E) \cdot f)$.  
\end{lem}

\begin{proof}
Assume that $E$ is $\mu$-semi-stable with respect to an ample divisor $H$ and
$E$ is not $\mu$-semi-stable with respect to an ample divisor $H'$.
Then there is a $\mu$-semi-stable subsheaf $E_1$ of $E$ satisfying
\begin{enumerate}
\item
$E_2:=E/E_1$ is $\mu$-semi-stable,
\item
$((\mu(E_1)-\mu(E_2)) \cdot H) \leq 0$ and $((\mu(E_1)-\mu(E_2)) \cdot H') > 0$.
\end{enumerate}
By Hodge index theorem, we get $((\mu(E_1)-\mu(E_2))^2)<0$.
By \cite[Lem. 2.1]{chamber},
$$
\Delta(E)=\frac{\rk(E_1)}{\rk(E)}\Delta(E_1)+\frac{\rk(E_2)}{\rk(E)}\Delta(E_2)-
\frac{\rk(E_1)\rk(E_2)}{2\rk(E)^2}((\mu(E_1)-\mu(E_2))^2).
$$
Since $\Delta(E_1) \geq 0$ and $\Delta(E_2) \geq 0$,
we have $((\mu(E_1)-\mu(E_2))^2) \geq 0$. Therefore $E$ is $\mu$-semi-stable for all
ample divisors. 
For an ample divisor $C_0+nf$ $(n \gg 0)$,
there is no $\mu$-semi-stable sheaf $E$ if $r \nmid (c_1(E) \cdot f)$. 
Therefore the claim holds.
\end{proof}

\begin{lem}\label{lem:mu-s}
We assume that $(K_X \cdot H)<0$. 
\begin{enumerate}
\item
Assume that $r \nmid (\xi \cdot f)$. Then
$$
\dim({\cal M}_H({\bf e})^{\mu ss} \setminus {\cal M}_H({\bf e})^{\mu s})<\dim {\cal M}_H({\bf e})^{\mu s}
$$
for a general $H$.  
\item
Assume that $r \mid (\xi \cdot f)$ and $\Delta>0$.
Then
$$
\dim({\cal M}_H({\bf e})^{\mu ss} \setminus {\cal M}_H({\bf e})^{\mu s})<\dim {\cal M}_H({\bf e})^{\mu s}
$$
for a general $H$.  
\end{enumerate}
\end{lem}

\begin{proof}
For a $\mu$-semi-stable sheaf $E$,
let 
$F:0 \subset F_1 \subset F_2 \subset \cdots \subset F_s=E$ 
be a Jordan-H\"{o}lder filtration of a $\mu$-semi-stable sheaf $E$
with respect to $\mu$-stability.
We set $E_i=F_i/F_{i-1}$.
Then 
$$
-\chi(E_i,E_j)=\rk(E_i)\rk(E_j)(g-1+\Delta(E_i)+\Delta(E_j)).
$$
For the first case, Lemma \ref{lem:Delta} implies
$\Delta(E_i)>0$ for all $i$.
Hence $-\sum_{i<j}\chi(E_i,E_j)>0$.

For the second case, by 
$$
\rk E \Delta(E)=\sum_i \rk E_i \Delta(E_i)
$$
(\cite[Lem. 2.1]{chamber}),
$\Delta(E_i)>0$ for an integer $i$.
Hence $-\sum_{i<j}\chi(E_i,E_j)>0$.
\end{proof}

\begin{lem}\label{lem:non-locallyfree}
We assume that $(K_X \cdot H)<0$. Then 
$\dim ({\cal M}_H({\bf e})^{\mu s} \setminus {\cal M}_H({\bf e})_0^{\mu s})=\dim {\cal M}_H({\bf e})^{\mu s}-(r-1)$.
\end{lem}

\begin{prop}\label{prop:irred}
We assume that $(K_X \cdot H)<0$. 
If $r \nmid (\xi \cdot f)$ or $\Delta>0$,
then ${\cal M}_H({\bf e})^{\mu s}_0$ is an open dense substack of ${\cal M}_H({\bf e})^{\mu ss}$.
In particular ${\cal M}_H({\bf e})^{\mu ss}$ is irreducible.
\end{prop}

\begin{proof}
By Lemma \ref{lem:mu-s} and Lemma \ref{lem:non-locallyfree}, we get the first claim.
Then the irreducibility follows from Corollary \ref{cor:Beauville}.
\end{proof}

\begin{rem}
If $\Delta=0$, then
there is no wall (Lemma \ref{lem:Delta}). Hence all $E \in {\cal M}_H({\bf e})^{\mu ss}$ are pull-back of
vector bundles on $C$, which shows the irreducibility of ${\cal M}_H({\bf e})^{\mu ss}$.
\end{rem}

\begin{NB}
 Let $W_x$ be a wall containing $H(x):=C_0+xf$ and let ${\cal C}_x^+$ (resp. ${\cal C}_x^-$) a chamber 
containing $H(x+\epsilon)$ (resp. $H(x-\epsilon)$) with $0< \epsilon \ll 1$.
For $(\gamma{}_1,\cdots,\gamma{}_s) \in \Gamma{}_{H_x,{\cal C}_x^-}$, 
we shall prove that $d_{\gamma{}_1,\cdots,\gamma{}_s} \geq 2$. 
Since $(\mu{}_j-\mu{}_i)^2<0$ and $\Delta{}_i \geq 0$, 
it is enough to prove that $r_ir_j(\mu{}_j-\mu{}_i,K_X/2) \geq 1$ for $i <j$.
We denote $r_ir_j(\mu{}_j-\mu{}_i)$ by $aC_0-bg$, and then $a$ and $b$ are positive integer.
A simple calculation shows that $r_ir_j(\mu{}_j-\mu{}_i,K_X/2)=b+ae/2 $. 
Therefore $d_{\gamma{}_1,\cdots,\gamma{}_s} \geq 1$.
In particular, if ${\cal M}_{H_x}^{\gamma{}}$ has $n$-irreducible components, then  
${\cal M}_{{\cal C}_x^-}^{\gamma{}}$ also has $n$-irreducible components.
\end{NB}

\subsection{Non-emptiness for $g=1$.}

 \begin{prop}\label{prop:1}
Assume that $g=1$.
For ${\bf e}=(r,\xi,\chi)$ with $0<(\xi \cdot f)<r$,
there exists a $\mu$-semi-stable sheaf $E$ of $\tau(E)={\bf e}$ with respect to $H(x)$ 
if and only if $x \leq \frac{e}{2}+\frac{r^2}{r_1r_2} \Delta({\bf e})$,
where $r_1:=(\xi \cdot f)$ and $r_2:=r-r_1$.
\end{prop}

\begin{proof}
If $e>0=2g-2$, then the claim is proved in \cite[Prop. 4.1]{chamber}.
So we assume that $e=0,-1$. 
\begin{NB}
$(K_X \cdot C)=e$ and $(K_X \cdot H(x))=e-2x$.
Hence $((K_X+f) \cdot H(x))<0$ if and only if
$x>(e+1)/2$.
Hence if $x>(e+1)/2$, then the claim is contained in 
\cite{chamber}. 
\end{NB}
We set $\xi=r_1 C_0+df$.
Let $E$ be a vector bundle of rank $r$  
defined by the following exact sequence
\begin{equation}\label{eq:gen}
 0 \to F_1(C_0) \to E \to F_2 \to 0,
\end{equation}
where $F_1$ (resp. $F_2$) is the pull-back of a semi-stable vector bundle of rank $r_1$ (resp. $r_2$) on $C$ 
with degree $d_1=r_1 e-d+\chi-(r+r_1)(1-g)$ (resp. $d-d_1$).
Then 
\begin{equation}
\begin{split}
\chi(E)=& \chi(F_1(C_0))+\chi(F_2)\\
=& d+d_1+(r+r_1)(1-g)-r_1 e=\chi.
\end{split}
\end{equation}
Hence we get 
$\tau(E)=(r,r_1 C_0+df,\chi)={\bf e}$.
We set
$$
x_0:=\frac{e}{2}+\frac{r^2}{r_1 r_2}\Delta(E)=e+\frac{1}{r_1 r_2}(r_1 d-r d_1).
$$
Then 
$((\mu(F_1(C_0))-\mu(F_2)) \cdot H(x))=0$ if and only if
$x=x_0$.
Hence $E \in {\cal M}_{H(x_0)}({\bf e})^{\mu ss}$.
Thus ${\cal M}_{H(x_0)}({\bf e})^{\mu ss}$ is not empty.
By Proposition \ref{prop:N(x)},
${\cal M}_{H(x)}({\bf e})^{\mu ss}$ is also nonempty for $x<x_0$.

Conversely assume that there is a $\mu$-semi-stable sheaf $E$ with respect to $H(y)$ such that
$\tau(E)={\bf e}$. By Lemma \ref{lem:Delta}, $\Delta({\bf e})>0$.
We set $x_0:=\frac{e}{2}+\frac{r^2}{r_1 r_2}\Delta({\bf e})$.
Then $H(x_0)$ is an ample divisor, since $e \leq 0$.
We have an irreducible component of ${\cal M}_{H(x_0)}({\bf e})^{\mu ss}$
parameterizing $E$ fitting in the exact exact sequence
\eqref{eq:gen}.
Assume that $y>x_0$.
By Proposition \ref{prop:N(x)},
$N(y) \subsetneq N(x_0)$. 
Hence ${\cal M}_{H(x)}({\bf e})^{\mu ss}$ has at least two irreducible components for $x \leq x_0$.
Therefore $y \leq x_0$. 
\begin{NB}
If $E_{|\pi^{-1}(\eta)} \cong {\cal O}_{\pi^{-1}(\eta)}(1)^{\oplus r_1} \oplus
 {\cal O}_{\pi^{-1}(\eta)}^{\oplus r_2}$, then
$\Ext^2(E,E(-f))=0$. Hence
$E$ deforms $E_{|\pi^{-1}(t)} \cong  {\cal O}_{\pi^{-1}(t)}(1)^{\oplus r_1} \oplus
 {\cal O}_{\pi^{-1}(t)}^{\oplus r_2}$ for all $t \in C$.
\end{NB}
\end{proof}

\begin{rem}
By the proof of Proposition \ref{prop:1},
a general member $E \in {\cal M}_{H(x)}({\bf e})^{\mu ss}$ is a locally free sheaf 
such that $\Ext^1(E_{|\pi^{-1}(t)},E_{|\pi^{-1}(t)})=0$ for a general $t \in C$.
\end{rem}

\end{document}